\documentclass[11pt,reqno]{amsart}
\usepackage{amsmath,amssymb,amsthm}
\usepackage{graphicx}
\usepackage[all]{xypic}
\usepackage[utf8,nocaptions]{vietnam}
\usepackage{calrsfs}
\input xypic
\def\thmsection{section}
\def\thmchangesection{changesection}

\def\thmchangechapter{changechapter}
\def\thmchange{change}
\def\thmplain{plain}
\ifx\theoremnumberstyle\thmplain
  \theoremstyle{break-italic}
  \newtheorem{satz}{Satz}
\else
  \ifx\theoremnumberstyle\thmsection
    \theoremstyle{break-italic}
    \newtheorem{satz}{Satz}[section]
  \else
    \ifx\theoremnumberstyle\thmchange
      \swapnumbers
      \theoremstyle{break-italic}
      \newtheorem{satz}{Satz}
    \else
      \ifx\theoremnumberstyle\thmchangesection
         \swapnumbers
         \theoremstyle{break-italic}
         \newtheorem{satz}{Satz}[section]
      \else
        \ifx\theoremnumberstyle\thmchangechapter
           \swapnumbers
           \theoremstyle{break-italic}
           \newtheorem{satz}{Satz}[chapter]
        \else
          \ifx\theoremnumberstyle\thmsection
             \theoremstyle{break-italic}
             \newtheorem{satz}{Satz}[section]
          \else
            \theoremstyle{break-italic}
            \newtheorem{satz}{Satz}[section]
          \fi
        \fi
      \fi
    \fi
  \fi
\fi

\theoremstyle{break-italic}
\newtheorem{theorem}[satz]{Theorem}
\newtheorem{lemma}[satz]{Lemma}

\newtheorem{corollary}[satz]{Corollary}
\newtheorem{Proposition}[satz]{Proposition}

\newtheorem*{conjecture*}{Conjecture}

\theoremstyle{break-roman}
\newtheorem{definition}[satz]{Definition}

\newtheorem{example}[satz]{Example}

\newtheorem{remark}[satz]{Remark}

\newtheorem{conjecture}[satz]{Conjecture}

\theoremstyle{standard}

\newtheorem*{claim}{Claim}

\theoremstyle{varthm-roman}
\newtheorem*{varthm-roman}{}
\theoremstyle{varthm-italic}
\newtheorem*{varthm-italic}{}
\theoremstyle{varthm-roman-break}
\newtheorem*{varthm-roman-break}{}
\theoremstyle{varthm-italic-break}
\newtheorem*{varthm-italic-break}{}
\theoremstyle{varthm-roman-no-punctuation}
\newtheorem{varthm-roman-no-punctuation-numbered}[satz]{}
\theoremstyle{varthm-italic-no-punctuation}
\newtheorem{varthm-italic-no-punctuation-numbered}[satz]{}

\newenvironment{varthm-roman-numbered}[1]{
  \begin{varthm-roman-no-punctuation-numbered}
    \mbox{\rm\textbf{#1}}
  }{\end{varthm-roman-no-punctuation-numbered}}
\newenvironment{varthm-italic-numbered}[1]{
  \begin{varthm-italic-no-punctuation-numbered}
    \mbox{\rm\textbf{#1}}
  }{\end{varthm-italic-no-punctuation-numbered}}
\newenvironment{varthm-roman-break-numbered}[1]{
  \begin{varthm-roman-no-punctuation-numbered}
    \mbox{\rm\textbf{#1}\newline}
  }{\end{varthm-roman-no-punctuation-numbered}}
\newenvironment{varthm-italic-break-numbered}[1]{
  \begin{varthm-italic-no-punctuation-numbered}
    \mbox{\rm\textbf{#1}}\newline
  }{\end{varthm-italic-no-punctuation-numbered}}
\numberwithin{equation}{section}

\def\ex{\begin{example}
  }
  \def\eex{\end{example}}
\def\thr{\begin{theorem}}
\def\ethr{\end{theorem}}
\def\pro{\begin{Proposition}}
\def\epro{\end{Proposition}}
\def\coro{\begin{corollary}}
\def\ecoro{\end{corollary}}
\def\df{\begin{definition}}
\def\edf{\end{definition}}
\def\lm{\begin{lemma}}
\def\elm{\end{lemma}}
\def\pf{\begin{proof}}
\def\epf{\end{proof}}
\def\problem{\begin{problem}}
\def\eproblem{\end{problem}}

\def\ite{\begin{itemize}}
\def\hite{\end{itemize}}
\def\rem{\begin{remark}}
\def\erem{\end{remark}}

\def\cla{\begin{claim}}
\def\ecla{\end{claim}}
\def\conj{\begin{conjecture}}
\def\econj{\end{conjecture}}

\def\eex{{\accent"5E e}\kern-.385em\raise.2ex\hbox{\char'23}\kern-.08em}
\def\EES{{\accent"5E E}\kern-.5em\raise.8ex\hbox{\char'23 }}
\def\ow{o\kern-.42em\raise.82ex\hbox{
\vrule width .12em height .0ex depth .075ex \kern-0.16em \char'56}\kern-.07em}
\def\OW{O\kern-.460em\raise1.36ex\hbox{
\vrule width .13em height .0ex depth .075ex \kern-0.16em \char'56}\kern-.07em}

\def\hoa{\mathcal}

\begin{document}

\title[Applications of Scherer-Hol's theorem ]{Some applications of Scherer-Hol's theorem for polynomial matrices}

\author{Trung Hoa Dinh}
\address{Trung Hoa Dinh, Department of Mathematics, Troy University, Troy, AL 36082, United States}
\email{thdinh@troy.edu}

\author{Toan Minh Ho }
\address{Minh Toan Ho, Institute of Mathematics, VAST, 18 Hoang Quoc Viet, Hanoi, Vietnam}
\email{hmtoan@math.ac.vn}

\author{Cong Trinh Le }
\address{Department of Mathematics, Quy Nhon University, 170 An Duong Vuong, Quy Nhon, Binh Dinh, Vietnam}
\email{lecongtrinh@qnu.edu.vn}


\date{\today}



\begin{abstract} In this paper we establish some applications of the Scherer-Hol's theorem  for polynomial matrices. Firstly, we give a representation for polynomial matrices positive definite on subsets of compact polyhedra.  Then we establish a Putinar-Vasilescu Positivstellensatz for homogeneous and non-homogeneous polynomial matrices. Next  we propose  a matrix version of the P\'olya-Putinar-Vasilescu Positivstellensatz. Finally, we approximate positive semi-definite polynomial matrices using sums of squares.
\end{abstract}

\maketitle
\section{Introduction}

Let $\mathbb R [X]:=\mathbb R[X_1,\ldots,X_n]$ denote the (commutative) algebra  of polynomials in $n$ variables $X_1,\ldots, X_n$ with real coefficients. For a fix integer $t>0$, we denote by $\mathcal{M}_t(\mathbb R[X])$ the algebra of $t\times t$ matrices with entries in $\mathbb R [X]$, and by $\mathcal{S}_t(\mathbb R[X])$ the subalgebra of symmetric matrices. Each element $\bold{A} \in  \hoa{M}_t(\mathbb R[X])$ is a matrix whose entries are polynomials in $\mathbb R[X]$, which is called  a \emph{polynomial matrix}.

For every subset $\mathcal{G}$ of $ \mathcal{S}_t(\mathbb R[X])$ we associate to the set 
$$K(\mathcal{G}):=\{x\in \mathbb R^{n}| \bold{G}(x) \geq 0, \forall \bold{G}\in \mathcal{G}\}. $$
Here the notation $\bold{G}(x) \geq 0$ means that the matrix $\bold{G}(x)$ is positive semi-definite, i.e. $v^T\bold{G}(x)v \geq 0$ for every vector $v\in \mathbb R^t$.  For $x\in \mathbb R^n$, the notation $\bold{G}(x)>0$ means that the matrix $\bold{G}(x)$ is positive  definite, i.e. $v^T\bold{G}(x)v > 0$ for every vector $v\in \mathbb R^t\setminus \{0\}$.\\
In particular, for a subset $G$ of $\mathbb R[X]$, 
$$ K(G)=\{x\in \mathbb R^n| g(x) \geq 0, \forall g\in G\}. $$

A result which represents positive polynomials on $K(G)$ is called a \textit{Positivstellensatz}.  P\'olya's Positivstellensatz (1928) represents homogenoeus polynomials  which are positive  on the orthant $\mathbb R_+^n\setminus \{0\}$.  Another Positivstellensatz "with denominators" was given by Krivine (1964) and Stengle (1974), which yields also a  proof for Artin's theorem on Hilbert's $17^{th}$ problem.   The first  "denominator-free" Positivstellensatz was discovered  by Schm\"udgen (1991, \cite{Schm4}). Some other "denominator-free" Positivstellens\"atze were followed by Putinar (1993, \cite{Pu}),  Schweighofer (2006, \cite{Schw}), etc. 

Handelman's Positivstellensatz (1988) represents positive  polynomials  on convex, compact polyhedra with non-empty interiors. Putinar and Vasilescu (1999, \cite{PV}) proposed a Positivstellensatz  for polynomials positive on $K(G)\setminus \{0\}$. Dickinson and Povh (2015, \cite{DP}) combined the P\'olya and the Putinar-Vasilescu theorems to establish a representation for  homogeneous polynomials positive on the intersection  $ \mathbb R_+^n \cap K(G) \setminus \{0\}$, which is called the \textit{P\'olya-Putinar-Vasilescu Positivstellensatz} in this paper.

A result which represents non-negative polynomials on $K(G)$ is called a \textit{Nichtnegativstellensatz}. A Nichtnegativstellensatz "with denominator" was given also by  Krivine (1964) and Stengle (1974).  Some other Nichtnegativstellens\"atze were discovered by Scheiderer (\cite{Sch1, Sch2}). In particular, Marshall (2003, \cite{Mar}) approximated non-negative polynomials on $K(G)$ using sums of squares.

A   version of   P\'{o}lya's   Positivstellensatz for polynomial matrices  was given by Scherer and Hol (2006, \cite{SchH}), with  applications e.g. in robust polynomial semi-definite programs.  Schm\"udgen's theorem for operator polynomials was discovered  by Cimprič and Zalar \cite{CZ}.  Handelman's Positivstellensatz for polynomial matrices was  studied in \cite{LeB}. Some other  Positivstellens\"atze for polynomial matrices were studied in \cite{Le}, with matrix  denominators.

A version of Putniar's Positivstellensatz for polynomial matrices was also given by Scherer and Hol (\cite{SchH}), see also in \cite[Theorem 13]{KSchw}.

\begin{theorem}\label{Scherer-Hol} Let $\mathcal{Q} \subseteq \mathcal{S}_t(\mathbb R[X])$ be an Archimedean quadratic module and $\bold{F} \in \mathcal{S}_t(\mathbb R[X])$.   If $\bold{F}(x) >0$ for all $x\in K(\mathcal{Q})$, then $\bold{F}\in \mathcal{Q}$.
\end{theorem}

A direct consequence of the Scherer-Hol theorem is the following
\begin{corollary} \label{coroScherer-Hol} Let $\mathcal{Q} \subseteq \mathcal{S}_t(\mathbb R[X])$ be an Archimedean quadratic module and $\bold{F} \in \mathcal{S}_t(\mathbb R[X])$.   If $\bold{F}(x) \geq 0$ for all $x\in K(\mathcal{Q})$, then $\bold{F}+\epsilon \bold{I}\in \mathcal{Q}$ for all $\epsilon >0$.
\end{corollary}

The main aim of this paper is to apply  the Scherer-Hol theorem  (Theorem \ref{Scherer-Hol} and its consequence, Corollary \ref{coroScherer-Hol})  to establish some Positivstellens\"atze (resp. Nichtnegativstellens\"atze) for polynomial matrices. More precisely,   we establish firstly in Section 3  a representation for polynomial matrices  positive definite on subsets of compact polyhedra. Next, in Section 4 we establish a Putinar-Vasilescu Positivstellensatz for homogeneous and non-homogeneous polynomial matrices, which also yields  a matrix version of Reznick's Positivstellensatz.  We propose in Section 5  a matrix version of the P\'olya-Putinar-Vasilescu Positivstellensatz. Finally, in Section 6 we propose a version of the Marshall theorem for polynomial matrices, approximating positive semi-definite polynomial matrices using sums of squares.

\section{Preliminaries}
In this section we shall recall some basis concepts and  facts in Real algebraic geometry for matrices over  commutative rings which are  proposed by Schm\"{u}dgen (\cite{Schm1}, \cite{Schm2}, \cite{Schm3}) and Cimprič (\cite{Ci1}, \cite{Ci2}).  

Let $\mathbb R [X]:=\mathbb R[X_1,\ldots,X_n]$ denote the (commutative) algebra  of polynomials in $n$ variables $X_1,\ldots, X_n$ with real coefficients. For a fix integer $t>0$, we denote by $\mathcal{M}_t(\mathbb R[X])$ the algebra of $t\times t$ matrices with entries in $\mathbb R [X]$, and by $\mathcal{S}_t(\mathbb R[X])$ the subalgebra of symmetric matrices. Each element $\bold{A} \in  \hoa{M}_t(\mathbb R[X])$ is a matrix whose entries are polynomials in $\mathbb R[X]$, which is called a \emph{polynomial matrix}.  $\bold{A}$ is also called a \emph{matrix polynomial},  because it can be viewed as a polynomial in $X_1,\ldots, X_n$ whose coefficients come from $\hoa{M}_t(\mathbb R)$. Namely, we can write $\bold{A}$ as 
$$  \bold{A}=\sum_{|\alpha|=0}^{d} \bold{A}_{\alpha}X^{\alpha}, $$
where $\alpha=(\alpha_1,\cdots,\alpha_n) \in \mathbb N_0^{n}$, $|\alpha|:=\alpha_1+\ldots + \alpha_n$, $X^{\alpha}:=X_1^{\alpha_1}\ldots X_n^{\alpha_n}$, $\bold{A}_\alpha \in \hoa{M}_t(\mathbb R)$, $d$ is the maximum over all degree of the entries of $\bold{A}$ and it is called the \textit{degree} of  the polynomial matrix $\bold{A}$.   To unify notation, throughout the paper each element of $\hoa{M}_t(\mathbb R[X])$ is called a \emph{polynomial matrix}. 

A subset $\mathcal{M}$ of $\hoa{S}_t(\mathbb R [X])$ is called a \textit{quadratic module}  if 
$$ \bold{I} \in \mathcal{M}, \quad \mathcal{M} + \mathcal{M} \subseteq \mathcal{M}, \quad  \bold{A}^T \mathcal{M} \bold{A} \subseteq \mathcal{M}, \forall \bold{A}\in \mathcal{M}_t(\mathbb R[X]).$$
The smallest quadratic module which contains a given subset $\mathcal{G}$ of $\mathcal{S}_t(\mathbb R[X])$ will be denoted by $\mathcal{M}({\mathcal{G}})$. It is clear that 
$$ \mathcal{M}(\mathcal{G}) = \{\sum_{i=1}^r\sum_{j=1}^s\bold{A}_{ij}^T \bold{G}_i \bold{A}_{ij}| r, s\in \mathbb N_0, \bold{G}_i \in \mathcal{G}\cup \{\bold{I}\}, \bold{A}_{ij}\in \mathcal{M}_t(\mathbb R[X])\}. $$
Each element of the form $\bold{A}^T\bold{A}$ is called a \textit{square} in $\mathcal{M}_t(\mathbb R[X])$. The set of all finite sums of squares in $\mathcal{M}_t(\mathbb R[X])$ is denoted by $\sum_t  \mathbb R[X]^2$.   Note that   $\mathcal{M}({\emptyset}) = \sum_t  \mathbb R[X]^2$.

In particular, a subset $M\subseteq \mathbb R[X]$ is called a quadratic module if 
$$1\in M, ~M+M\subseteq M, ~a^2M\subseteq M ~\forall a\in \mathbb R[X].$$
 The smallest quadratic module of $\mathbb R[X]$ which contains a given subset $G\subseteq  \mathbb R[X]$ will be denoted by $M(G)$, and it consists of all  elements  of the form $\sigma_0+\sum_{i=1}^m\sigma_ig_i$, where $m \in \mathbb N$, $g_i\in G,$ and $\sigma\in\sum\mathbb R[X]^2$-the set of finite sums of squares of polynomials in $\mathbb R[X]$. \\
A subset $M\subseteq \mathbb R[X]$ is said to be a \textit{semiring} if 
$$M+M\subseteq M, ~ MM\subseteq M, ~ \mathbb R_{\geq 0} \subseteq  M.$$
For $G=\{g_1,\ldots,g_m\}\subseteq \mathbb R[X]$,  \textit{the semiring generated by $G$} consists of finite sums of  terms of the form 
$$ a_\alpha g_1^{\alpha_1}\ldots g_m^{\alpha_m}, \quad \alpha=(\alpha_1,\ldots, \alpha_m) \in \mathbb N_0^m, a_\alpha \geq 0,$$
and denoted by $P(G)$. 

For a quadratic module  or a semiring $M$ in $ \mathbb R[X]$, denote 
$$ M^t:=\{\sum_{i} m_i \bold{A}_i^T \bold{A}_i | m_i\in M, \bold{A}_i\in \hoa{M}_t( \mathbb R[X])\}. $$
Since $M^t$ contains the set of sums of squares in $\hoa{M}_t(\mathbb R[X])$,  $M^t$ is always  a \textit{quadratic module} on $\hoa{M}_t( \mathbb R[X])$. 

For any matrix $\bold{A}\in \mathcal{M}_t(\mathbb R[X])$, the  notation $\bold {A} \geq 0$ means $\bold{A}$ is \emph{positive semidefinite},  i.e. for each ${x} \in \mathbb R^{n}$, $v^{T}\bold{A}(x)v \geq 0$   for all  $v\in \mathbb R^t$; $\bold {A} >0 $ means $\bold{A}$ is \emph{positive definite}, i.e. for each  ${x} \in \mathbb R^{n}$, $v^{T}\bold{A}(x)v > 0$   for all  $v\in \mathbb R^t \setminus\{0\}$. 

We associate each set $\mathcal{G}\subseteq \mathcal{S}_t(\mathbb R[X])$ to the set 
$$K(\mathcal{G}):=\{x\in \mathbb R^{n}| \bold{G}(x) \geq 0, \forall \bold{G}\in \mathcal{G}\}, $$
which is a basic closed semi-algebraic set in $\mathbb R^n$. 
In particular, for a subset $G$ of $\mathbb R[X]$, 
$$ K(G)=\{x\in \mathbb R^n| g(x) \geq 0, \forall g\in G\}. $$
The following result of Cimprič (\cite{Ci2}) shows that the set ${K}({\hoa{G}})$ can be determined by \emph{scalars}, i.e. by polynomials in $\mathbb R[X]$. 

\begin{lemma}[{\cite[Proposition 5]{Ci2}}] \label{lm-cimpric} Let $\mathcal{G}\subseteq \mathcal{S}_t(\mathbb R[X])$. Then there exists a subset $G$ of $ \mathbb R[X] $ with the following properties:
\begin{itemize}
\item[(1)] ${K}({\mathcal{G}}) = K({G})$;
\item[(2)] $M(G)^{t} \subseteq \mathcal{M}({\hoa{G}}).$
\end{itemize}
Moreover, if $\mathcal{G}$ is finite then $G$ can be chosen to be finite. On the other hand, if $\mathcal{G}$ consists of homogeneous polynomial matrices, then the polynomials in $G$ are also homogeneous. 
\end{lemma} 

A quadratic module or a semiring $Q$ on $\mathbb R[X]$ (resp. $\mathcal{M}_t(\mathbb R[X])$) is said to be \textit{Archimedean} if for every $f\in \mathbb R[X]$ (resp. $\bold{F}\in \mathcal{M}_t(\mathbb R[X])$), there exists a $\lambda >0$ such that $\lambda \pm f \in Q$ (resp. $\lambda \cdot \bold{I} \pm \bold{F} \in Q$).

\begin{lemma}[{\cite[Lemma 12.7, Coro. 12.8]{Schm2}}] \label{Archimed-semiring}Let $Q $ be a quadratic module or a semiring on $\mathbb R[X_1,\ldots,X_n]$. Then $Q$ is Archimedean if and only if there exists $\lambda > 0$ such that $\lambda \pm X_i \in Q$, for all $i=1,\ldots,n$. 

Moreover, if $Q$ is a quadratic module, then $Q$ is Archimedean if and only if there exists $\lambda > 0$ such that $\lambda - \sum_{i=1}^n X_i^2 \in Q$.
\end{lemma}

\begin{lemma}\label{Archimed-Mn}  Let $M$ be a quadratic module or a semiring on $\mathbb R[X]$. Then $M$ is Archimedean if and only if $M^t$ is Archimedean. Moreover, for a finite subset $G$ of $\mathbb R[X]$, we have 
\begin{equation}
K\big(M(G)^t\big)=K\big(M(G)\big)=K(G)=K\big(P(G)\big)=K\big(P(G)^t\big).
\end{equation}
\end{lemma}
\begin{proof}
For the case $M$ is a quadratic module, the result follows from  \cite[Prop. 4]{Le}. If $M$ is a semiring, the result follows from Lemma \ref{Archimed-semiring}. The latter equalities are straightforward. 
\end{proof}

\section{Polynomial matrices positive definite on subsets of compact polyhedra}
In this section we give an application of the Scherer-Hol theorem   to  represent polynomial matrices which are positive definite on subsets of compact polyhedra.

Let $m$ and $k$ be positive integers with $m\leq k$. Let 
$$G=\{g_1,\ldots, g_k\} \subseteq \mathbb R[X]:=\mathbb R[X_1,\ldots,X_n]$$ such that $g_1,\cdots,g_m$ are linear. Denote $\hat{G}=\{g_1,\ldots,g_m\}$. Note that $K(G) \subseteq K(\hat{G})$. Let $P(G)$ be the semiring generated by $G$. The following result is a matrix version of \cite[Theorem 12.44]{Schm2}.

\begin{theorem} \label{compact-polyhedra} Suppose that $K(\hat{G})$ is non-empty and compact. For  $\bold{F}\in  \mathcal{S}_t(\mathbb R[X])$, if $\bold{F}(x)>0$ for all $x\in K(G)$, then $\bold{F} \in P(G)^t$, i.e. $\bold{F}$ can be written as
$$\bold{F}=\sum_{i=1}^r \Big(\sum_{j=1}^s  a_{\alpha_{ij}} g^{\alpha_{ij}}\Big)\bold{A}_i^T\bold{A}_i, $$
with $\alpha_{ij}\in \mathbb N_0^k$,  $a_{\alpha_{ij}} \geq 0$, $g^{\alpha_{ij}}:=g_1^{(\alpha_{ij})_1}\ldots g_k^{(\alpha_{ij})_k}$ and $\bold{A}_i \in \mathcal{M}_t(\mathbb R[X])$. 
\end{theorem}
\begin{proof} Since $K(\hat{G})$ is  compact, there exists  $\lambda > 0$ such that for each $i=1,\ldots, n$, the linear polynomial $\lambda \pm X_i $ is non-negative on  $K(\hat{G})$. Since $K(\hat{G})$ is non-empty, it follows from an affine form of  \textit{Farkas' lemma} (cf. \cite[Lemma 12.43]{Schm3}) that  for each $i=1,\ldots,n$ we have 
$$ \lambda \pm X_i = \lambda_0 + \lambda_1f_1 + \ldots + \lambda_m f_m,$$
with  $\lambda_j \geq 0$, $j=1,\ldots,m$. Hence $ \lambda \pm X_i \in P(G)$ for all $i=1,\ldots, n$. By Lemma \ref{Archimed-semiring}, the semiring $P(G)$ is Archimedean.

Moreover, since $P(G)^t$ contains the set of sums of squares $\sum_t \mathbb R[X]^2$, it is a   quadratic module on $\mathcal{M}_t(\mathbb R[X])$. It follows from  Lemma \ref{Archimed-Mn} that $P(G)^t$ is also Archimedean and
$$ K\big(P(G)^t\big)= K(P(G))=K(G).$$
For each $x\in {K}\big(P(G)^t\big)$, we have $x\in K(G)$, hence 
$\bold{F}(x) > 0 $. It follows from the Scherer-Hol theorem that $\bold{F}\in P(G)^t$. The proof is complete.
\end{proof}

\section{A Putinar-Vasilescu Positivstellensatz for polynomial matrices}
The Putinar-Vasilescu Positivstellensatz for homogeneous polynomials is stated as follows.
\begin{theorem}[{\cite[Theorem 4.5]{PV}}] Let $f$ and $g_1,\ldots,g_m$ be homogeneous polynomials in $\mathbb R[X]:=\mathbb R[X_1,\ldots,X_n]$ of even degree. Denote  $G=\{g_1,\ldots,g_m\}$. If $f(x) > 0$ for all $x\in K(G)\setminus \{0\}$, then there exists a number $N>0$ such that 
$$ (\sum_{i=1}^n X_i^2)^N f \in M(G). $$
\end{theorem}

In this section we apply the Scherer-Hol  theorem to give a matrix version of this Positivstellensatz. 

\begin{theorem} \label{Putinar-Vasilescu} Let $\mathcal{G}\subseteq \mathcal{M}_t(\mathbb R[X])$ be a finite set of homogeneous polynomial matrices of even degrees.   Let $\bold{F}\in  \mathcal{S}_t(\mathbb R[X])$ be a homogeneous polynomial matrix of even degree $d>0$.  If $\bold{F}(x)>0$ for all   $x\in K(\mathcal{G})\setminus \{0\}$, then there exist a finite set $G$ of homogeneous polynomials   in $\mathbb R[X]$ of even degrees  and a number $N>0$ such that 
$$ (\sum_{i=1}^n X_i^2)^N \bold{F} \in M(G)^t \subseteq \mathcal{M}(\mathcal{G}). $$ 
\end{theorem}
\begin{proof} It follows from Lemma \ref{lm-cimpric} that there exists a finite subset $ G=\{g_1,\ldots,g_m\}$ of $\mathbb R[X]$ consisting  of homogeneous polynomials  of even degrees $d_1,\ldots, d_m$, respectively, such that 
$$ K(G) = K(\mathcal{G}) \mbox{ and } M(G)^t \subseteq \mathcal{M}(\mathcal{G}).$$
 Let $\lambda >0$ such that $K(G)\cap \mathbb  S(0;\lambda^2) \not = \emptyset$, where $\mathbb S:=\mathbb S(0;\lambda^2)$ denotes the sphere 
$$ \{x\in \mathbb R^n : \lambda^2 - \sum_{i=1}^n x_i^2 =0\}. $$
Denote 
$$G'=G\cup \{\lambda^2 - \sum_{i=1}^n X_i^2, \sum_{i=1}^n X_i^2 - \lambda^2\}. $$
Then $K(G')=K(G)\cap \mathbb S$, and $M(G')=M(G)+\left<\lambda^2-\sum_{i=1}^n X_i^2\right>$, where $\left<\lambda^2-\sum_{i=1}^n X_i^2\right>$ denotes the ideal in $\mathbb R[X]$ generated by the polynomial $\lambda^2-\sum_{i=1}^n X_i^2$. 

Since $\lambda^2-\sum_{i=1}^n X_i^2 \in M(G')$, it follows from Lemma \ref{Archimed-semiring} that $M(G')$ is an Archimedean quadratic module. Then it follows from Lemma \ref{Archimed-Mn} that the quadratic module $M(G')^t$ is also Archimedean on $\mathcal{M}_t(\mathbb R[X])$.  By Lemma \ref{Archimed-Mn},
$$ K\big(M(G')^t\big)=K(M(G')) = K(G')=K(G)\cap \mathbb S.$$
For any $x \in  K\big(M(G')^t\big)=K(G')$, we have $x\in K(G)\cap \mathbb S$, hence $x\in K(G)\setminus \{0\}$. Then $ \bold{F}(x) >0. $ It follows from the Scherer-Hol theorem that $\bold{F} \in  M(G')^t$, i.e.  $\bold{F}$ can be expressed as 
\begin{align}
\bold{F}(X)&=\sum_{i=1}^l \big(\sigma_{i0}(X) + \sigma_{i1}(X)g_1(X) + \ldots + \sigma_{im}(X)g_m(X)\big)\bold{A}^T_i(X)\bold{A}_i(X)+ \nonumber \\
& + \sum_{i=1}^l h_i(X)(\lambda^2 - \sum_{j=1}^nX_j^2)\bold{A}^T_i(X)\bold{A}_i(X), \label{equ1}
\end{align}
where $\sigma_{ij}\in \sum \mathbb R[X]^2$, $h_i\in \mathbb R[X]$, $\bold{A}_i \in \mathcal{M}_t(\mathbb R[X])$. 

Substituting each $X_i$ by $\dfrac{\lambda X_i}{\sqrt{\sigma}}$ in both sides of (\ref{equ1}), where $\sigma: = \sum_{j=1}^n X_j^2$, observing that $$\lambda^2 - \sum_{j=1}^n\big(\dfrac{\lambda X_i}{\sqrt{\sigma}}\big)^2 =0,  $$
$$ \bold{F}\big(\dfrac{\lambda X}{\sqrt{\sigma}}\big) = \dfrac{\lambda^d}{\sigma^{d/2}}\bold{F}(X), \mbox{ and } g_j\big(\dfrac{\lambda X}{\sqrt{\sigma}}\big) = \dfrac{\lambda^{d_j}}{\sigma^{d_j/2}}g_j(X),$$
we have 
\begin{equation} \label{equ4.1}
\dfrac{\lambda^d}{\sigma^{d/2}}\bold{F}(X)=\sum_{i=1}^l \big(\sigma_{i0}\big(\dfrac{\lambda X}{\sqrt{\sigma}}\big) + \sum_{j=1}^m\dfrac{\lambda^{d_j}}{\sigma^{d_j/2}}\sigma_{ij}\big(\dfrac{\lambda X}{\sqrt{\sigma}}\big)g_j(X)\big)\bold{A}^T_i\big(\dfrac{\lambda X}{\sqrt{\sigma}}\big)\bold{A}_i\big(\dfrac{\lambda X_i}{\sqrt{\sigma}}\big).
\end{equation}
Denote 
\begin{align*}
e_1&:= \max\{\deg(\sigma_{ij}), j=0,\ldots,m\},\\
e_2&:= \max\{d_j, j=1,\ldots,m\},\\
e_3&:= \max\{\deg({\bold{A}_i}), i=1,\ldots, l\},
\end{align*}
which are even numbers. Put $N:=d/2+e_1/2+e_2/2+e_3$, and multiplying both sides of (\ref{equ4.1}) for $\sigma^N$, we have
\begin{align*} \label{equ4.2}
\lambda^d \sigma^{N-d/2}\bold{F}(X)&=\sigma^{d/2}\sum_{i=1}^l \Bigg(\sigma^{e_1/2+e_2/2}\sigma_{i0}\big(\dfrac{\lambda X}{\sqrt{\sigma}}\big) + \\
& + \sum_{j=1}^m\lambda^{d_j}(\sigma^{e_1/2}\sigma_{ij}(\dfrac{\lambda X}{\sqrt{\sigma}}))\sigma^{e_2/2-d_j/2}g_j(X)\Bigg)\sigma^{e_3}\bold{A}^T_i\big(\dfrac{\lambda X}{\sqrt{\sigma}}\big)\bold{A}_i\big(\dfrac{\lambda X_i}{\sqrt{\sigma}}\big).
\end{align*}
Note that 
$$ \sigma'_{i0}:= \sigma^{e_1/2+e_2/2}\sigma_{i0}\big(\dfrac{\lambda X}{\sqrt{\sigma}}\big) \mbox{ and } \sigma'_{ij}:= \lambda^{d_j}(\sigma^{e_1/2}\sigma_{ij}(\dfrac{\lambda X}{\sqrt{\sigma}}))\sigma^{e_2/2-d_j/2}$$
are sums of squares in $\mathbb R[X]$;
$$\bold{B}_i:=\sigma^{e_3/2}\bold{A}_i\big(\dfrac{\lambda X}{\sqrt{\sigma}}\big) \in  \mathcal{M}_t(\mathbb R[X]). $$
Then 
$$ \sigma^{N-d/2}\bold{F} = \sum_{i=1}^l \Big(\theta_{i0} + \sum_{j=1}^m \theta_{ij}g_j\Big)\bold{B}_i^T \bold{B}_i,  $$
where $\theta_{ij}:=\lambda^{-d}\sigma^{d/2}\sigma'_{ij} \in \sum \mathbb R[X]^2$. It follows that 
$$ \sigma^{N-d/2}\bold{F}\in M(G)^t\subseteq \mathcal{M}(\mathcal{G}).$$
\end{proof}

In the case $\mathcal{G}=\emptyset$, we have the following matrix version of \textit{Reznick's Positivstellensatz}.
\begin{corollary}
Let $\bold{F}\in  \mathcal{S}_t(\mathbb R[X])$ be a homogeneous polynomial matrix.  If $\bold{F}(x)>0$ for all   $x\in \mathbb R^n\setminus \{0\}$, then there exists a number $N>0$ such that $ (\sum_{i=1}^n X_i^2)^N \bold{F} \in \sum_t \mathbb R[X]^2. $
\end{corollary}

To give a non-homogeneous version of Theorem \ref{Putinar-Vasilescu}, we need the following notions. For a polynomial 
$$g(X)=\sum_{|\alpha|\leq e} g_\alpha X^\alpha \in \mathbb R[X_1,\ldots,X_n]$$ of  degree $e$, its \textit{homogenization} in the ring  $\mathbb R[X_0,X_1,\ldots,X_n]$ is defined by 
$$ \tilde{g}(X_0,X_1,\ldots,X_n):=\sum_{|\alpha|\leq e} g_\alpha X^\alpha X_0^{e-|\alpha|}.$$
 It is clear that $\tilde{g}$ is homogeneous of degree $e$ and $\tilde{g}(1,x_1,\ldots,x_n)=g(x_1,\ldots,x_n)$ for all $(x_1,\ldots,x_n)\in \mathbb R^n$..
 
 For a polynomial matrix $\bold{G}\in  \mathcal{M}_t(\mathbb R[X_1,\ldots,X_n])$ of degree $d$, we can write 
 $$ \bold{G}(X)=\sum_{|\alpha|\leq d}\bold{G}_\alpha X^\alpha, $$
 with $\bold{G}_\alpha \in \mathcal{M}_t(\mathbb R)$. Its homogenization in the algebra $\mathcal{M}_t(\mathbb R[X_0,X_1,\ldots,X_n])$ is defined by
 $$ \widetilde{\bold{G}}(X_0,\ldots,X_n)=\sum_{|\alpha|\leq d}\bold{G}_\alpha X^\alpha X_0^{d-|\alpha|}. $$
 It is obvious that $\widetilde{\bold{G}}$ is homogeneous of degree $d$ and $ \widetilde{\bold{G}}(1,x_1,\ldots,x_n) =   \bold{G} (x_1,\ldots,x_n)$ for all $(x_1,\ldots,x_n)\in \mathbb R^n$.

\begin{corollary} \label{Putinar-Vasilescu-nonhomo} Let $\mathcal{G}\subseteq \mathcal{M}_t(\mathbb R[X])$ be a finite set of polynomial matrices of even degrees.   Let $\bold{F}\in  \mathcal{S}_t(\mathbb R[X])$ be a  polynomial matrix of even degree.  Denote $\widetilde{\mathcal{G}}:=\{\widetilde{\bold{G}}| \bold{G} \in \mathcal{G}\}\subseteq \mathcal{M}_t(\mathbb R[X_0,X_1,\ldots,X_n]).$
If $\widetilde{\bold{F}}(x)>0$ for all   $x\in K(\widetilde{\mathcal{G}})\setminus \{0\}$, then there exist a finite set $G$ of  polynomials  in $\mathbb R[X]$ of even degrees   and a number $N>0$ such that 
$$ (1+\sum_{i=1}^n X_i^2)^N \bold{F} \in M(G)^t \subseteq \mathcal{M}(\mathcal{G}). $$ 
\end{corollary}
\begin{proof} It follows from Theorem \ref{Putinar-Vasilescu} that  there exist a finite set $\widetilde{G}$ of homogeneous polynomials  of even degrees  in $\mathbb R[X_0,X_1,\ldots,X_n]$ and a number $N>0$ such that 
\begin{equation}\label{equ-putinar-vasilescu-homo}
(\sum_{i=0}^n X_i^2)^N \widetilde{\bold{F}} \in M(\widetilde{G})^t \subseteq \mathcal{M}(\widetilde{\mathcal{G}}). 
\end{equation} 
Denote $G=\{g(1,X_1,\ldots,X_n) | g \in \widetilde{G}\}$. Since $M(\widetilde{G})^t \subseteq \mathcal{M}(\widetilde{\mathcal{G}})$, we have $M(G)^t \subseteq \mathcal{M}(\mathcal{G})$. Substituting $X_0=1$ in both sides of (\ref{equ-putinar-vasilescu-homo}) we obtain
$$(1+\sum_{i=1}^n X_i^2)^N  \bold{F}  \in M(G)^t \subseteq \mathcal{M}(\mathcal{G}).  $$
\end{proof}

\section{A P\'olya-Putinar-Vasilescu Positivstellensatz for polynomial matrices }
Dickinson and Povh (2015, \cite[Theorem 3.5]{DP}) proved the following Positivstellensatz, which is so-called the \textit{P\'olya-Putinar-Vasilescu Positivstellensatz} for homogeneous polynomials, stated as follows.

\begin{theorem} Let $f$ and $g_1,\ldots,g_m$ be homogeneous polynomials in $\mathbb R[X]$ of even degree.  Denote $G=\{g_1,\ldots,g_m\}$. If $f(x) > 0$ for all $x\in \mathbb R_+^n\cap K(G)\setminus \{0\}$, then there exists a number $N>0$ and homogeneous polynomials $h_i, i=1,\ldots, m$ with nonnegative coefficients  such that 
$$ (\sum_{i=1}^n X_i)^N f = \sum_{i=1}^m h_ig_i. $$
\end{theorem}

In this section we apply the Scherer-Hol theorem  to establish a version of this Positivstellensatz for homogeneous polynomial matrices.
\begin{theorem} \label{Polya-Putinar-Vasilescu} Let $\mathcal{G}\subseteq \mathcal{M}_t(\mathbb R[X])$ be a finite set of homogeneous polynomial matrices of even degrees.   Let $\bold{F}\in  \mathcal{S}_t(\mathbb R[X])$ be a homogeneous polynomial matrix of even degree $d>0$.  If $\bold{F}(x)>0$ for all   $x\in \mathbb R_{+}^n \cap K(\mathcal{G})\setminus \{0\}$, then there exist a   set $G=\{g_1,\ldots,g_m\}\subseteq \mathbb R[X]$ consisting of  homogeneous polynomials  of even degrees, a number $N>0$, homogeneous polynomials $h_{\alpha_{ij}}$ with nonnegative coefficients,   and polynomial matrices $\bold{A}_i \in  \mathcal{M}_t(\mathbb R[X])$, for $ i=1,\ldots, l; j=1,\ldots, r$, such that    
$$ (\sum_{i=1}^n X_i)^N \bold{F} = \sum_{i=1}^l \Big(\sum_{j=1}^r h_{\alpha_{ij}}g^{\alpha_{ij}}\Big)\bold{A}_i^T\bold{A}_i, $$ 
where $\alpha_{ij}\in \mathbb N_0^m $, $g^{\alpha_{ij}}:=g_1^{(\alpha_{ij})_1}\ldots g_m^{(\alpha_{ij})_m}$.
\end{theorem}

To give a proof for this Positivstellensatz, we need the following results for semirings in $\mathbb R[X]$.

Let $P_0$ be the set of all polynomials in $\mathbb R[X]$ with nonnegative coefficients. For $G=\{g_1,\ldots,g_m\}\subseteq \mathbb R[X]$, denote by $P(G)$ the semiring in $\mathbb R[X]$ generated by $G$. Put
$$ P_0P_G :=\Big\{\sum_{i=1}^r h_{\alpha_i} g_1^{(\alpha_i)_1}\ldots g_m^{(\alpha_i)_m}|r\in \mathbb N_0, \alpha_i \in \mathbb N_0^m, h_{\alpha_i} \in P_0\Big\}.$$
Let $\lambda >0$ such that $K(G)\cap \{\lambda - \sum_{i=1}^n X_i =0\} \not = \emptyset$. Denote 
$$G':=G\cup \{X_1,\ldots,X_n\}\cup \{\lambda - \sum_{j=1}^n X_j, \sum_{j=1}^n X_j - \lambda\}.$$
Let $P(G')$ be the semiring in $\mathbb R[X]$ generated by $G'$. 
\begin{lemma} \label{lm5.2} $P(G')=P_0P(G) + \left<\lambda - \sum_{j=1}^n X_j\right>$.
\end{lemma}
\pf Since each element of $P(G')$ is a finite sum of elements of the form
$$a_{\alpha\beta\gamma} X_1^{\alpha_1}\ldots X_n^{\alpha_n} g_1^{\beta_1}\ldots g_m^{\beta_m} (\lambda - \sum_{j=1}^n X_j)^{\gamma_1}(\sum_{j=1}^nX_j -\lambda)^{\gamma_2},$$
with $a_{\alpha\beta\gamma} \geq 0, \alpha_i, \beta_j, \gamma_k\in \mathbb N_0$, we have $P(G')\subseteq P_0P(G) + \left<\lambda - \sum_{j=1}^n X_j\right>$.

Conversely, since $P_0P(G) \subseteq P(G')$, it is sufficient to prove that 
$$\left<\lambda - \sum_{j=1}^n X_j\right> \subseteq P(G').$$
 In fact, for each polynomial $p \in \mathbb R[X]$, we have 
$$p = p_+ - p_-,$$
where $p_+$ and $p_-$ are in $P_0$. Since $\lambda - \sum_{j=1}^n X_j \in P(G')$ and $\sum_{j=1}^n X_j - \lambda \in P(G')$, it is easy to verify that for every $p (\lambda - \sum_{j=1}^n X_j) \in \left<\lambda - \sum_{j=1}^n X_j\right>$ with  $p\in \mathbb R[X]$, we have
$$ p (\lambda - \sum_{j=1}^n X_j) = p_+(\lambda - \sum_{j=1}^n X_j) + p_-(\sum_{j=1}^n X_j-\lambda) \in P(G'). $$
The proof is complete.

\epf

\begin{lemma} \label{lm5.3} $P(G')$ is an Archimedean semiring, hence $P(G')^t$ is an Archimedean quadratic module in $\mathcal{M}_t(\mathbb R[X])$.
\end{lemma}
\begin{proof} For each $i=1,\ldots, n$, since $X_i \in  P(G')$ and $\lambda >0$, we have 
$$\lambda + X_i \in P(G'). $$
Moreover, we have 
$$\lambda - X_i = (\lambda - \sum_{i=1}^n X_i) + \sum_{i=2}^n X_i  \in P(G').$$
It follows from Lemma \ref{Archimed-semiring} that $P(G')$ is an Archimedian semiring. 
\end{proof}

\begin{proof}[\textbf{Proof of Theorem \ref{Polya-Putinar-Vasilescu}}] It follows from Lemma \ref{lm-cimpric} that there exists a finite subset $ G=\{g_1,\ldots,g_m\}$ of $\mathbb R[X]$ consisting  of homogeneous polynomials  of even degrees $d_1,\ldots, d_m$, respectively, such that 
$$ K(G) = K(\mathcal{G}) \mbox{ and } M(G)^t \subseteq \mathcal{M}(\mathcal{G}).$$
Let $\lambda >0$ such that $K(G)\cap \{\lambda - \sum_{i=1}^n X_i =0\} \not = \emptyset$. Denote 
$$G':=G\cup \{X_1,\ldots,X_n\}\cup \{\lambda - \sum_{j=1}^n X_j, \sum_{j=1}^n X_j - \lambda\}.$$
Let $P(G')$ be the semiring in $\mathbb R[X]$ generated by $G'$. It follows from Lemma \ref{lm5.2} that
$$P(G')=P_0P(G) + \left<\lambda - \sum_{j=1}^n X_j\right>,$$
and by Lemma \ref{Archimed-Mn}, we have
$$ K\big(P(G')^t\big)= K(P(G'))= K(G')= \mathbb R_+^n\cap K(G)\cap \{\lambda - \sum_{k=1}^n X_k=0\}.$$
Then, for each $x \in K\big(P(G')^t\big)$, we have $x\in \mathbb R_+^n\cap K(G)\cap \{\lambda - \sum_{k=1}^n X_k=0\}$, hence 
 $x\in R_+^n\cap K(G)\setminus \{0\}$. The hypothesis implies that 
 $ \bold{F}(x)  >0.$ Note that $P(G')^t$ is Archimedean by Lemma \ref{lm5.3}. Thus, applying the Scherer-Hol theorem we obtain 
$$\bold{F} \in P(G')^t=\Big(P_0P(G)+\left<\lambda - \sum_{k=1}^n X_k\right>\Big)^t.$$
Then $\bold{F}$ can be written as 
\begin{equation} \label{equ5.1} \bold{F} =\sum_{i=1}^l  \Big(\sum_{j=1}^r h'_{\alpha_{ij}}g^{\alpha_{ij}} + \varphi_i(\lambda - \sum_{k=1}^n X_k)\Big)\bold{B}_i^T\bold{B}_i, 
\end{equation}
with $\alpha_{ij}\in \mathbb N_0^m$, $h'_{\alpha_{ij}}\in P_0$, $g^{\alpha_{ij}}:=g_1^{(\alpha_{ij})_1}\ldots g_m^{(\alpha_{ij})_m}$, $\varphi_i\in \mathbb R[X]$,  $\bold{B}_i \in \mathcal{M}_t(\mathbb R[X])$.

Substituting each $X_i$ by $\dfrac{\lambda X_i}{\sigma}$ in both sides of (\ref{equ5.1}), where $\sigma: = \sum_{k=1}^n X_k$, observing that $$\lambda - \sum_{k=1}^n\dfrac{\lambda X_k}{\sigma} =0,  $$
$$ \bold{F}\big(\dfrac{\lambda X}{\sigma}\big) = \dfrac{\lambda^d}{\sigma^{d}}\bold{F}(X), \mbox{ and } g^{\alpha_{ij}}\big(\dfrac{\lambda X}{\sigma}\big) = \dfrac{\lambda^{p_{ij}}}{\sigma^{p_{ij}}}g^{\alpha_{ij}}(X),$$
where $p_{ij}=(\alpha_{ij})_1d_1+\ldots+(\alpha_{ij})_md_m$,  
we have 
\begin{equation}\label{equ5.2}
\dfrac{\lambda^d}{\sigma^{d}}\bold{F}(X) = \sum_{i=1}^l  \Big(\sum_{j=1}^r h'_{\alpha_{ij}}\big(\dfrac{\lambda X}{\sigma}\big)\dfrac{\lambda^{p_{ij}}}{\sigma^{p_{ij}}}g^{\alpha_{ij}}(X)\Big)\bold{B}_i^T\big(\dfrac{\lambda X}{\sigma}\big)\bold{B}_i\big(\dfrac{\lambda X}{\sigma}\big).
\end{equation}
Let 
\begin{align*}
e_1&:= \max\{\deg(h'_{\alpha_{ij}}), i=1,\ldots,l; j=1,\ldots,r\};\\
e_2&:= \max\{p_{ij}, i=1,\ldots,l; j=1,\ldots,r\};\\
e_3&:=\max\{\deg(\bold{B}_i),   i=1,\ldots,l\}.
\end{align*}
Put $N:=d +e_1 +e_2 +2e_3$, and multiplying both sides of (\ref{equ5.2}) with $\sigma^N$, we get
\begin{align*} 
 \lambda^d \sigma^{N-d} \bold{F}(X) &=  \sum_{i=1}^l  \Big(\sum_{j=1}^r\Big(\sigma^{d+e_1+e_2}\dfrac{\lambda^{p_{ij}}}{\sigma^{p_{ij}}}h'_{\alpha_{ij}}\big(\dfrac{\lambda X}{\sigma}\big)\Big)g^{\alpha_{ij}}(X)\Big)\cdot \\
 & \cdot \Big(\sigma^{e_3}\bold{B}_i^T\big(\dfrac{\lambda X}{\sigma}\big)\Big)\Big(\sigma^{e_3}\bold{B}_i\big(\dfrac{\lambda X}{\sigma}\big)\Big).
\end{align*}
Note that $\bold{A}_i:=\lambda^{-d}\sigma^{e_3}\bold{B}_i\big(\dfrac{\lambda X}{\sigma}\big) \in \mathcal{M}_t(\mathbb R[X])$. Moreover, consider the polynomial 
$$ h''_{\alpha_{ij}}(X)=\sigma^{d+e_1+e_2}\dfrac{\lambda^{p_{ij}}}{\sigma^{p_{ij}}}h'_{\alpha_{ij}}\big(\dfrac{\lambda X}{\sigma}\big). $$
For any $\mu \in \mathbb R, \mu \not =0$, we have 
$$h''_{\alpha_{ij}}(\mu X) =\mu^{d+e_1+e_2-p_{ij}}\sigma^{d+e_1+e_2}\dfrac{\lambda^{p_{ij}}}{\sigma^{p_{ij}}}h'_{\alpha_{ij}}\big(\dfrac{\lambda X}{\sigma}\big) = \mu^{d+e_1+e_2-p_{ij}}h''_{\alpha_{ij}}(X). $$
It follows that $h''_{\alpha_{ij}}$ is a homogeneous polynomial of degree $d+e_1+e_2-p_{ij}$. Since $h'_{\alpha_{ij}}$ has nonnegative coefficients, so does  $h''_{\alpha_{ij}}$. Denote $h_{\alpha_{ij}}=\dfrac{h''_{\alpha_{ij}}}{\lambda^d}$. Then $h_{\alpha_{ij}}$ is homogeneous with nonnegative coefficients, and 
$$ \sigma^{N-d} \bold{F} = \sum_{i=1}^l  \Big(\sum_{j=1}^r h_{\alpha_{ij}} g^{\alpha_{ij}}\Big)\bold{A}_i^T \bold{A}_i. $$
This completes the proof.
\end{proof}

 In the case $\mathcal{G}=\emptyset$, we have the following matrix version of the  \textit{P\'olya Positivstellensatz}. 
\begin{corollary}
Let $\bold{F}\in  \mathcal{S}_t(\mathbb R[X])$ be a homogeneous polynomial matrix of even degree $d$.  If $\bold{F}(x)>0$ for all   $x\in \mathbb R_+^n\setminus \{0\}$, then there exists a number $N>0$, homogeneous polynomials $h_{i}$ with nonnegative coefficients and polynomial matrices $\bold{A}_i \in  \mathcal{M}_t(\mathbb R[X])$, for $ i=1,\ldots, l$, such that    
$$ (\sum_{i=1}^n X_i)^N \bold{F} = \sum_{i=1}^l h_i\bold{A}_i^T\bold{A}_i. $$ 
\end{corollary}
\begin{proof}
The result  follows from the proof of Theorem \ref{Polya-Putinar-Vasilescu}, with the fact that when $\mathcal{G}=\emptyset$, we have $G=\emptyset$ and $P(\emptyset)=\mathbb R_{\geq 0} $ - the set of non-negative real numbers, and  $P(G')=P_0+\left<\lambda - \sum_{k=1}^n X_k\right>$.
\end{proof}

In the following we give a non-homogeneous version of the  P\'olya-Putinar-Vasilescu Positivstellensatz for polynomial matrices, whose proof is similar to that of Corollary \ref{Putinar-Vasilescu-nonhomo}.
\begin{corollary}
Let $\mathcal{G}\subseteq \mathcal{M}_t(\mathbb R[X])$ be a finite set of polynomial matrices of even degrees.   Let $\bold{F}\in  \mathcal{S}_t(\mathbb R[X])$ be a  polynomial matrix of even degree.  Denote $\widetilde{\mathcal{G}}:=\{\widetilde{\bold{G}}| \bold{G} \in \mathcal{G}\}\subseteq \mathcal{M}_t(\mathbb R[X_0,X_1,\ldots,X_n]).$
If $\widetilde{\bold{F}}(x)>0$ for all   $x\in \mathbb R_{+}^{n+1} \cap K(\widetilde{\mathcal{G}})\setminus \{0\}$, then there exist a finite set $G=\{g_1,\ldots,g_m\}\subseteq \mathbb R[X]$ consisting of    polynomials  of even degrees, a number $N>0$,   polynomials $h_{\alpha_{ij}}$ with nonnegative coefficients, and polynomial matrices $\bold{A}_i \in  \mathcal{M}_t(\mathbb R[X])$, for $ i=1,\ldots, l; j=1,\ldots, r$, such that    
$$ (1+\sum_{i=1}^n X_i)^N \bold{F} = \sum_{i=1}^l \Big(\sum_{j=1}^r h_{\alpha_{ij}}g^{\alpha_{ij}}\Big)\bold{A}_i^T\bold{A}_i, $$ 
where  $\alpha_{ij}\in \mathbb N_0^m $, $g^{\alpha_{ij}}:=g_1^{(\alpha_{ij})_1}\ldots g_m^{(\alpha_{ij})_m}$.
\end{corollary}

\section{Approximating positive semi-definite polynomial matrices using sums of squares }
Marshall (2003) proved the following theorem, which approximates non-negative polynomials on basic closed semi-algebraic sets.
\begin{theorem}[{\cite[Coro. 4.3]{Mar}}]
Let $G$ be a finite subset of $\mathbb R[X]:=\mathbb R[X_1,\ldots,X_n]$ and $f\in \mathbb R[X]$. The following are equivalent:
\begin{itemize}
\item[(1)] $f(x)\geq 0$ for every $x\in K(G)$.
\item[(2)] There exists an integer  $k\geq 0 $ such that for all rational $\epsilon >0$, there exists an integer $l\geq 0$ satisfying $p^l(f+\epsilon p^k) \in M(G)$, where $p=1+\displaystyle\sum_{i=1}^nX_i^2$.
\end{itemize}
\end{theorem}
In this section we give a matrix version of this theorem, approximating positive semi-definite polynomial matrices using sums of squares. The first version is established for homogeneous polynomial matrices, as follows.

\begin{theorem}\label{Marshall}
Let $\mathcal{G}\subseteq \mathcal{M}_t(\mathbb R[X])$ be a finite set of homogeneous polynomial matrices of even degrees.   Let $\bold{F}\in  \mathcal{S}_t(\mathbb R[X])$ be a homogeneous polynomial matrix of even degree $d>0$.  If $\bold{F}(x)\geq 0$ for all   $x\in K(\mathcal{G})$, then there exist a finite set $G$ of homogeneous polynomials in $\mathbb R[X]$ of even degrees   and a number $\lambda >0$ such that for every $\epsilon >0$, there exists a number $N>0$  satisfying 
$$ \sigma^{N-d/2} (\bold{F}+\dfrac{\epsilon}{\lambda^d}\sigma^{d/2}\bold{I}) \in M(G)^t \subseteq \mathcal{M}(\mathcal{G}),$$ 
where $\sigma=\sum_{i=1}^n X_i^2$.
\end{theorem}
\begin{proof} The existence of the  set $G=\{g_1,\ldots,g_m\}$ of homogeneous polynomials in $\mathbb R[X]$ of even degrees $d_1,\ldots,d_m$, respectively,   satisfying $K(G)=K(\mathcal{G})$ and $M(G)^t\subseteq \mathcal{M}(\mathcal{G})$ is given in the proof of Theorem \ref{Putinar-Vasilescu}.

Let $\lambda >0$ such that $K(G)\cap \mathbb  S \not = \emptyset$. 
Denote 
$$G'=G\cup \{\lambda^2 - \sum_{i=1}^n X_i^2, \sum_{i=1}^n X_i^2 - \lambda^2\}. $$
Then $K(G')=K(G)\cap \mathbb S$, and $M(G')=M(G)+\left<\lambda^2-\sum_{i=1}^n X_i^2\right>$ which is Archimedean.  Then the quadratic module $M(G')^t$ is also Archimedean, and  
$$ K\big(M(G')^t\big)=K(M(G')) = K(G')=K(G)\cap \mathbb S.$$
For any $x \in  K\big(M(G')^t\big)$, we have $x\in K(G)\cap \mathbb S$, hence $x\in K(G)$. Then $ \bold{F}(x) \geq 0. $ It follows from Corollary \ref{coroScherer-Hol}  that for every $\epsilon >0$,  $\bold{F} + \epsilon \bold{I} \in  M(G')^t$, i.e.  $\bold{F} + \epsilon \bold{I}$ can be expressed as 
\begin{align}
\bold{F} + \epsilon \bold{I}&=\sum_{i=1}^l \big(\sigma_{i0}(X) + \sum_{j=1}^m \sigma_{ij}(X)g_j(X)\big)\bold{A}^T_i(X)\bold{A}_i(X)+ \nonumber \\
& + \sum_{i=1}^l h_i(X)(\lambda^2 - \sum_{j=1}^nX_j^2)\bold{A}^T_i(X)\bold{A}_i(X), \label{equ6.1}
\end{align}
where $\sigma_{ij}\in \sum \mathbb R[X]^2$, $h_i\in \mathbb R[X]$, $\bold{A}_i \in \mathcal{M}_t(\mathbb R[X])$. 

Substituting each $X_i$ by $\dfrac{\lambda X_i}{\sqrt{\sigma}}$ in both sides of (\ref{equ6.1}), where $\sigma: = \sum_{j=1}^n X_j^2$, observing that $$\lambda^2 - \sum_{j=1}^n\big(\dfrac{\lambda X_i}{\sqrt{\sigma}}\big)^2 =0,  $$
$$ \bold{F}\big(\dfrac{\lambda X}{\sqrt{\sigma}}\big) = \dfrac{\lambda^d}{\sigma^{d/2}}\bold{F}(X), \mbox{ and } g_j\big(\dfrac{\lambda X}{\sqrt{\sigma}}\big) = \dfrac{\lambda^{d_j}}{\sigma^{d_j/2}}g_j(X),$$
we have 
\begin{equation} \label{equ6.2}
\dfrac{\lambda^d}{\sigma^{d/2}}\bold{F}(X) + \epsilon \bold{I}=\sum_{i=1}^l \big(\sigma_{i0}\big(\dfrac{\lambda X}{\sqrt{\sigma}}\big) + \sum_{j=1}^m\dfrac{\lambda^{d_j}}{\sigma^{d_j/2}}\sigma_{ij}\big(\dfrac{\lambda X}{\sqrt{\sigma}}\big)g_j(X)\big)\bold{A}^T_i\big(\dfrac{\lambda X}{\sqrt{\sigma}}\big)\bold{A}_i\big(\dfrac{\lambda X_i}{\sqrt{\sigma}}\big).
\end{equation}
Denote 
\begin{align*}
e_1&:= \max\{\deg(\sigma_{ij}), j=0,\ldots,m\},\\
e_2&:= \max\{d_j, j=1,\ldots,m\},\\
e_3&:= \max\{\deg({\bold{A}_i}), i=1,\ldots, l\},
\end{align*}
which are even numbers. Put $N:=d/2+e_1/2+e_2/2+e_3$, and multiplying both sides of (\ref{equ6.2}) for $\sigma^N$, we have
\begin{align*} \label{equ6.3}
\lambda^d \sigma^{N-d/2}\bold{F}(X)&+ \epsilon\sigma^N \bold{I}=\sigma^{d/2}\sum_{i=1}^l \Bigg(\sigma^{e_1/2+e_2/2}\sigma_{i0}\big(\dfrac{\lambda X}{\sqrt{\sigma}}\big) + \\
& + \sum_{j=1}^m\lambda^{d_j}(\sigma^{e_1/2}\sigma_{ij}(\dfrac{\lambda X}{\sqrt{\sigma}}))\sigma^{e_2/2-d_j/2}g_j(X)\Bigg)\sigma^{e_3}\bold{A}^T_i\big(\dfrac{\lambda X}{\sqrt{\sigma}}\big)\bold{A}_i\big(\dfrac{\lambda X_i}{\sqrt{\sigma}}\big).
\end{align*}
Since
$ \sigma'_{i0}:= \sigma^{e_1/2+e_2/2}\sigma_{i0}\big(\dfrac{\lambda X}{\sqrt{\sigma}}\big)$ and $\sigma'_{ij}:= \lambda^{d_j}(\sigma^{e_1/2}\sigma_{ij}(\dfrac{\lambda X}{\sqrt{\sigma}}))\sigma^{e_2/2-d_j/2}$ are sums of squares in $\mathbb R[X]$, and 
$\bold{B}_i:=\sigma^{e_3/2}\bold{A}_i\big(\dfrac{\lambda X}{\sqrt{\sigma}}\big) \in  \mathcal{M}_t(\mathbb R[X]),$ we have 
$$\sigma^{N-d/2}(\bold{F}+ \dfrac{\epsilon}{\lambda^d}\sigma^{d/2} \bold{I})=\sigma^{N-d/2}\bold{F}(X)+ \dfrac{\epsilon}{\lambda^d}\sigma^N \bold{I} \in M(G)^t\subseteq \mathcal{M}(\mathcal{G}).$$ 
The proof is complete.
\end{proof}
A non-homogeneous version of Theorem \ref{Marshall} is given as follows, whose proof is similar to that of Corollary \ref{Putinar-Vasilescu-nonhomo}.
\begin{corollary}\label{coroMarshall}
Let $\mathcal{G}\subseteq \mathcal{M}_t(\mathbb R[X])$ be a finite set of   polynomial matrices of even degrees.   Let $\bold{F}\in  \mathcal{S}_t(\mathbb R[X])$ be a   polynomial matrix of even degree $d>0$.  If $\widetilde{\bold{F}}(x)\geq 0$ for all   $x\in K(\widetilde{\mathcal{G}})$, then there exist a finite set $G$ of   polynomials  in $\mathbb R[X]$ of even degrees   and a number $\lambda >0$ such that for every $\epsilon >0$, there exists a number $N>0$  satisfying 
$$ (1+\sigma)^{N-d/2} (\bold{F}+\dfrac{\epsilon}{\lambda^d}(1+\sigma)^{d/2}\bold{I}) \in M(G)^t \subseteq \mathcal{M}(\mathcal{G}),$$ 
where $\sigma=\sum_{i=1}^n X_i^2$.
\end{corollary}

\subsection*{Acknowledgements} The third author would like to express his sincere  gratitude to Prof. Konrad Schm\"udgen for fruitful discussions on representation theory for the algebra of matrices. This paper was finished during the visit of the  second and the third authors   at the Vietnam Institute for Advanced Study in Mathematics (VIASM). They thanks VIASM for financial support and hospitality.


\end{document}